\documentclass[11pt,reqno]{amsart}

\usepackage[T1]{fontenc}
\usepackage[utf8]{inputenc}
\usepackage[a4paper,margin=29mm]{geometry}
\usepackage{lmodern}
\usepackage{microtype}
\microtypesetup{expansion=false}
\usepackage{amsmath,amssymb,amsthm}
\usepackage{mathtools}
\usepackage{mathrsfs}
\usepackage{enumitem}
\usepackage{tikz-cd}
\usepackage{quiver}
\usepackage{xcolor}
\usepackage{aliascnt}
\usepackage{url}
\usepackage{hyperref}
\usepackage[nameinlink,capitalize,noabbrev]{cleveref}

\hypersetup{
  colorlinks=true,
  linkcolor=blue!55!black,
  citecolor=green!40!black,
  urlcolor=blue!65!black
}

\numberwithin{equation}{section}

\newtheorem{theorem}{Theorem}[section]

\newaliascnt{proposition}{theorem}

\aliascntresetthe{proposition}

\newaliascnt{lemma}{theorem}
\newtheorem{lemma}[lemma]{Lemma}
\aliascntresetthe{lemma}

\newaliascnt{corollary}{theorem}

\aliascntresetthe{corollary}

\newaliascnt{assumption}{theorem}

\aliascntresetthe{assumption}

\theoremstyle{definition}

\newaliascnt{definition}{theorem}
\newtheorem{definition}[definition]{Definition}
\aliascntresetthe{definition}

\newaliascnt{example}{theorem}

\aliascntresetthe{example}

\newaliascnt{problem}{theorem}
\newtheorem{problem}[problem]{Open problem}
\aliascntresetthe{problem}

\theoremstyle{remark}

\newaliascnt{remark}{theorem}

\aliascntresetthe{remark}

\newaliascnt{warning}{theorem}

\aliascntresetthe{warning}

\crefname{proposition}{Proposition}{Propositions}
\Crefname{proposition}{Proposition}{Propositions}
\crefname{lemma}{Lemma}{Lemmas}
\Crefname{lemma}{Lemma}{Lemmas}
\crefname{corollary}{Corollary}{Corollaries}
\Crefname{corollary}{Corollary}{Corollaries}
\crefname{assumption}{Assumption}{Assumptions}
\Crefname{assumption}{Assumption}{Assumptions}
\crefname{definition}{Definition}{Definitions}
\Crefname{definition}{Definition}{Definitions}
\crefname{example}{Example}{Examples}
\Crefname{example}{Example}{Examples}
\crefname{remark}{Remark}{Remarks}
\Crefname{remark}{Remark}{Remarks}
\crefname{warning}{Warning}{Warnings}
\Crefname{warning}{Warning}{Warnings}

\usepackage[table]{xcolor}

\definecolor{lightgray}{gray}{0.9}

\title[Game Chromatic Number with Singletons and Large Parts]{The Game Chromatic number of complete multipartite graphs with singletons and large parts}
\author{Lasitha Jayasinghe}
\date{}

\begin{document}

\maketitle
\begin{abstract}
We study the game chromatic number of complete multipartite graphs whose parts are either singletons or are of size at least four. Building on the work of Pawe\l \   Obszarski et al., we give a formula in terms of the ordered part sizes, and the number of singleton parts.
\end{abstract}
\section{Introduction}
In 1981, a two-player non-cooperative map-coloring game was first conceived by Steven Brams and presented by Martin Gardner \cite{Gardner1981}. The game was later rediscovered by Bodlaender \cite{Bodlaender1991} in the context of graphs, and is called the \textit{graph coloring game}.

The graph coloring game is usually played between two players, Alice and Bob, on a graph $G$ with a set of $k$ fixed colors. The players take turns properly coloring (no two vertices of the same color are neighbours). Alice wins when all the vertices are properly colored using $k$ colors; otherwise, Bob wins.

We say Alice has a \textit{winning strategy} if she can garauntee completely coloring the graph irrespective of Bob's moves, and we say Bob has a winning strategy is he can ensure the graph cannot be completely colored irrespective of Alice's moves.

The smallest value of $k$ for which Alice has a winning strategy is known as the game chromatic number of $G$, denoted $\chi_g(G)$. Notably, since the game yields a proper coloring, $\chi_g(G) \ge \chi(G)$.
\subsection{Previous results}
The game chromatic number has been investigated for several graph classes, including forests \cite{Faigle1993, Toci2015, FURTADO2019461}, toroidal graphs \cite{RASPAUD20091183}, cartesian graphs \cite{BROSSE202546}, planar \cite{KiersteadTrotter1994, Zhu1999} and outerplanar \cite{GuanZhu1999} graphs, and complete multipartite graphs \cite{Dunn2012, Furtado2026, Obszarski2023}. On the other hand, Costa et al. \cite{COSTA202036} established that the decision problem relating to the game is PSPACE-complete for general graphs.

A \textit{complete} $k-$\textit{multipartite} graph is a connected graph whose vertices can be partitioned into $k$ independent sets (referred to as \textit{parts}). We refer to a part with exactly one vertex as a \textit{singleton}.

The graph coloring game was first studied in the context of complete multipartite graphs in passing by Dunn \cite{Dunn2012}, where the game chromatic number of complete multipartite graphs with all parts of equal size was established. Pawe\l \   Obszarski et al. \cite{Obszarski2023} later determined the game chromatic number for all complete multipartite graphs with no singleton. 

There was also a poster \cite{poster} in which there was a formula without any proof for complete $k-$partite graphs $K_{r_1,\dots , r_k}$ with either all $r_i \in \{1,2\}$ or all $r_i \in \{1,3\}$. The poster also stated that $\chi_g(K_{r_1,\dots , r_k})$ is known for all graphs with $r_i \in \{1,2,3\}$. The proofs of the former, and precise formula and proof of the latter were provided by Furtado, Dantas et al. \cite{Furtado2026}, very recently, along with a partial attempt at certain families of large part graphs. They explicitly leave determining an exact expression for singleton parts together with unequal large parts open.

In this paper, we make further progress towards a solution to the problem for all complete $k-$partite graphs. We work with complete multipartite graphs $K_{1^A, r_1,\dots , r_D}$ with $A \ge 1$ singleton parts and $D\ge 1$ non-singleton parts $V_1,\dots , V_D$ of sizes $r_1 \ge r_2 \ge \dots \ge r_D\ge 4$, and completely describe their game chromatic numbers by giving explicit strategies for Alice and Bob.

\subsection{Notation and concepts}

Throughout this paper, we refer to a part with at least four vertices as \textit{large}. Moreover, we will call a part with all vertices colored, no vertices colored or some (but not all) vertices colored a \textit{fully colored}, \textit{uncolored} or \textit{partially colored} part respectively.

We will also say that a player \textit{opened} a part if he or she colored the first vertex of this part. Let us also refer to the move after which every part has at least one colored vertex as the \textit{fixing move}. 

Observe that if Alice has a strategy such that she and Bob use at most $l$ colors by the fixing move, $l$ colors are sufficient to color the entire graph (colors can be repeated within a part without conflict) and $\chi_g(K) \le l$. Similarly, if Bob has a strategy which consumes $l$ colors before the fixing move, $l$ colors are insufficient to color the entire graph and $\chi_g(K)>l$.

Moreover, when we say one of the players \textit{picks} a vertex, we imply that Alice colors the vertex with a color already used in the same partition unless she is opening a partition and Bob always colors the vertex using a new color.

Let $M$ denote the number of distinct colors introduced up to and including the fixing move, and let $N$ denote the number of colored vertices by the fixing move.

Define $P=A+ \sum_{i=1}^D \left \lceil \frac{r_i}{2} \right \rceil$. Further, let $R_j=\sum_{i=1}^j r_i$ for each $1\le j \le D$ and $R_0=0$, and define 
	\[F_j=2A+2D-2j-3 + \left \lceil \frac{(R_j-A+1)_+}{2} \right \rceil\]
	where $x_+ = \max(x,0)$ for a real number $x$.
\subsection{Main results}
We provide a series of strategies for Alice, and reuse a single optimal strategy for Bob from \cite{Obszarski2023}. Using them, we completely describe the game chromatic numbers (see Table 1 below for a summary) of complete multipartite graphs with only singletons and large parts. 
\begin{table}[ht]
\centering
\renewcommand{\arraystretch}{1.5}
\caption{Summary of our results.}
\begin{tabular}{|c|c|c|c|c|}
\hline
$D$ & $A$ & $\chi_g(K_{r_1,\ldots,r_k})$
& Upper bound & Lower bound \\
\hline

1
& 1
& 2
& obvious
& obvious \\
\hline

\rowcolor{gray!15}

1 
& $A>1$ 
& $\min\{2A-1, A+\lceil \frac{r}{2} \rceil\}$
& Theorem 4.2 
& Theorem 4.2 \\
\hline

2 
& $A=r_1=r_2$ odd 
& 2A 
& Theorem 4.3 
& Theorem 4.3 \\
\hline

\rowcolor{gray!15} 

	2 & otherwise & $\min \{2A+1, P\}$ & Section 2.3 & Section 3.5 \\
\hline
	$\ge 3$ & any & $\min\{F_0,F_1,\dots , F_{D-2},P\}$ & Section 2.3 & Section 3.5 \\
	\hline
\end{tabular}
\end{table}
\subsection{Acknowledgements} 
We thank Evan Chen for comments on an earlier version of this manuscript, and the organizers of Math Beyond Limits, including Vuk Jankovi\'c, Jovan Vukovi\'c, Damjan Stankovi\'c, Radek \v{Z}ak,
Daniela Spurtacz, and Boris Stankovi\'c,  for bringing the problem to attention through their 2026 Qualifying Quiz.
\section{Strategies for Alice}
The common idea for Alice is to wait by picking singletons or repeating colors until a certain threshold is reached, and then `rush' by opening a previously uncolored part on every turn. We have two main strategies.
\subsection{Pairing}
\begin{definition}
	Let \textbf{(A4)} denote the following strategy for Alice:
	\begin{enumerate}
		\item in the first move pick a singleton,
		\item while an uncolored singleton remains, if Bob has just picked a vertex in a large part $P$ which is partially colored, color a vertex of $P$ using the same color otherwise color a singleton
	\item once all singletons are colored, open an uncolored large part on every turn until the fixing move
	\end{enumerate}
\end{definition}
Observe that each large part $V_i$ of size $r_i$ before Alice switches to rushing uses at most $\left \lceil \frac{r_i}{2} \right \rceil$ colors since Alice repeats a color after every move that leaves this part partially colored. Moreover, if there are $u$ large parts left when Alice switches to rushing, it takes at most $2u-1$ moves to open all of them and since $r_i \ge 4$, $\left \lceil \frac{r_i}{2} \right \rceil \ge 2$ for each of these parts.

In conclusion, at most
\[M \le A+ \sum_{i=1}^{D-u} \left \lceil \frac{r_i}{2} \right \rceil +2u \le A+\sum_{i=1}^{D} \left \lceil \frac{r_i}{2} \right \rceil \]
colors are used by the fixing move. Hence, $\chi_g(K_{1^A,r_1,\dots , r_D}) \le P$.

\subsection{Limited-reuse}
We first need to make a couple of algebraic observations. Let $q_j =\left \lceil \frac{(R_j-A+1)_+}{2} \right \rceil$. 
\begin{lemma}
	There exists an index $i$ for which $F_i = \min \{F_0,\dots , F_{D-2}\}$ and $q_i \le 1$.
\end{lemma}
\begin{proof}
Choose the smallest index $j$ attaining the minimum. If $j=0$ then $q_j=0$, so consider $j \ge 1$ in what follows. Assume contrariwise that $q_j \ge 2$. Now, 
	\[F_j-F_{j-1} = q_j-q_{j-1}-2\]
	which is non-negative when $q_{j-1}=0$. Otherwise, $q_{j-1} >0$, so
	\[q_j =\left \lceil \frac{R_j-A+1}{2} \right \rceil \ge \left \lceil \frac{R_{j-1}-A+1}{2} \right \rceil +2 = q_{j-1}+2\]
	which once again makes $F_j \ge F_{j-1}$, contradicting our assumption. Therefore, $q_j \le 1$.
\end{proof}
\begin{definition}
	Fix a minimizing index $j$ with $q_j \in \{0,1\}$. Let \textbf{(A5)} denote the following strategy for Alice:
	\begin{enumerate}
		\item in the first move pick a singleton,
		\item if Bob has opened $j+1$ parts (possibly including singletons), switch to rushing
		\item otherwise, if $q_j=1$ and the reuse is unused, and a partially colored large part $P$ exists, color a vertex of $P$ reusing a color
		\item otherwise color a singleton
	\end{enumerate}
\end{definition}
We first need to confirm that this strategy can actually be implemented. Assume contrariwise that at the beginning of Alice's turn, all singletons are colored, her $q:=q_j$ reuse moves are consumed, but Bob has opened at most $j$ parts. 

Let $s$ be the number of singletons colored by Bob, and $l$ the number of large parts opened by Bob. Then, $s+l \le j$, Alice has colored $A-s$ singletons and $q$ reuse moves. Since it is her turn, each player must have performed exactly $t=A-s+q$ moves at this point.

Of Bob's $t$ moves we know that $s$ are singleton moves, so the total number of vertices colored in large parts is exactly 
\[(t-s)+q = A-2s+2q.\]
However, since Alice has not opened any large parts, these must all lie in the $l$ large parts opened by Bob which contain a total of at most $R_l$ vertices. Moreover, since every large part has size at least four, 
\[R_l \le R_j - 4(j-l) \le R_j - 4s\]
From the definition of $q$, 
\[R_j \le A+2q-1\] 
so combining these three inequalities we conclude
\[A-2s+2q \le R_l \le A+2q-1-4s\]
implying $2s+1 \le 0$, a clear contradiction.

The only remaining possibility therefore, is for Alice to have failed to use her $q$ reuse moves (which requires $q=1$). From the definition of $q$ it follows that  $R_j \in \{A,A+1\}$. In particular, $j \ge 1$ so
\[A \ge R_j -1 \ge 4j-1 > 2j\]
Since Alice uses her reuse move as soon as possible, it must be the case that both players have only performed singleton moves until now. Since $A>2j$, if Bob has colored at most $j$ singletons, a singleton is still available, contradicting our assumption that Alice is stuck.

Now, consider the moment when Bob opens $j+1$ parts. Since $j \le D-2$, the fixing move has not already occurred, and Alice can rush through the remaining parts. If Alice performed $x \le q_j$ reuse moves, and opened $s$ singletons, she made a total of $x+s$ moves. We also have $A+D-(j+1)-s$ empty parts, so at most $2A+2D-2j-2s-3$ more colors are needed to color these. Bob received $x+s$ moves until this point, so a total of at most
\[M \le s + (x+s) + (2A+2D-2j-2s-3) = 2A+2D-2j-3 +x \le F_j\]
colors are used by the fixing move.
\subsection{Upper bound}
Combining the observations of the last two sections, we can conclude that by picking the optimal strategy for Alice, 
\[\chi_g(K_{1^A,r_1,\dots , r_D}) \le \min\{F_0,F_1,\dots , F_{D-2},P\}\]
for all $D \ge 3$.
\section{Strategies for Bob}

In this section we deal with $D \ge 2$, with the subsection 2.4 focusing on $D \ge 3$. Recall from \cite{Obszarski2023} the following simple strategy for Bob.
\begin{definition}
	Let \textbf{(B1)} denote the following strategy for Bob:
	\begin{enumerate}
		\item if the part $P$ Alice just colored a vertex from is partially colored, pick a vertex of $P$
		\item otherwise if the part $P$ Alice just colored a vertex from is fully colored, but there is some other partially colored part $P$, pick a vertex of $P$
		\item otherwise if there are no partially colored parts, pick a vertex from the largest uncolored part
	\end{enumerate}
\end{definition}

Now, say at the fixing move we have $a$ completely colored large parts opened by Alice and $b$ completely colored large parts opened by Bob. Then, $m=D-a-b$ large parts are partially colored.

\subsection{$m \ge 1$ and $b \le D-2$.}
Of the $m$ parts, at most one partially colored part was opened by Bob. Moreover, Bob did not pick any singleton. Hence, Alice opened at least $a+(m-1)=D-b-1$ large parts, forcing her to introduce a new color at least $A+D-b-1$ times. Bob gets at least $A+D-b-2$ moves in the meantime forcing at least
\begin{equation}
	M \ge 2A+2D-2b-3
\end{equation}
colors to be used by the fixing move.

We can further note that since \textbf{(B1)} prefers larger parts, at the fixing move, the completed large parts include the $b$ largest parts, up to ties. Thus, the completely colored large parts contain a total of at least $R_b$ vertices. Since a total of $a+b$ large parts were completely colored, there are at least $a$ other large parts, each with at least 4 vertices for a total of at least $4a$ more vertices in completely colored large parts.

As before, at most one of the $m$ parts was opened by Bob, and accounting for possibly the partition on which the fixing move was performed at least $m-2$ of the partially colored parts were opened by Alice, with Bob immediately coloring another vertex of that part. Hence, these parts contain a total of at least $2m-2$ colored vertices. 

In total, at least 
\[N \ge A+R_b + 4a + 2m-2 = A+2D+R_b+2a-2b-2\]
vertices were colored by the fixing move. Since Bob colors $\left \lfloor \frac{N}{2} \right \rfloor$ of these vertices, and uses a different color every time, we conclude that a total of at least
\begin{equation}
	M \ge A+D-b-1 + \left \lfloor \frac{N}{2} \right \rfloor \ge A+2D+a-2b-2 + \left \lfloor \frac{A+R_b}{2} \right \rfloor
\end{equation}
colors must be used by the fixing move.  

Observe that the bound $(3.2)$ exceeds $(3.1)$ by $-A+a+1+ \left \lfloor \frac{A+R_b}{2} \right \rfloor$. When this quantity is non-positive, the bound $(3.1)$ is preferred, and when it is positive $(3.2)$ is. In summary, 
\[M \ge 2A+2D-2b-3 + \left(-A+1+ \left \lfloor \frac{A+R_b}{2} \right \rfloor \right)_+ = 2A+2D-2b-3 + \left \lceil \frac{(R_b-A+1)_+}{2} \right \rceil\]

\subsection{$m \ge 1$ and $b=D-1$.} 
Note that we must have $m=1$ and $a=0$ by definition. The $D-1$ completed large parts have at least $R_{D-1}$ vertices, and the partially colored part has at least one colored vertex for a total of at least
\[N \ge A+R_{D-1}+1\]
colored vertices. Bob has $\left \lfloor \frac{N}{2} \right \rfloor$ moves, and Alice supplies the singleton colors implying that a total of at least
\[M \ge A+\left \lfloor \frac{A+R_{D-1}+1}{2} \right \rfloor\]
colors are used by the fixing move.
\subsection{$m=0$} 
The key is to note that for each large part $V_i$ of size $r_i$, the number of colors used in $V_i$ is at least $\left \lceil \frac{r_i}{2} \right \rceil$.

To see why, we split into cases based on who opened $V_i$. If Bob opened, for every Alice move Bob gets another move in the same partition except possibly for the last so the number of vertices Bob picks is at least the number of vertices Alice picks, implying that at least $\left \lceil \frac{r_i}{2} \right \rceil$ colors are used.

However, if Alice opened, Bob replies immediately with a different color, and for every Alice move, Bob gets another move in the same partition except possibly for the last. Hence, Bob can pick at least $\left \lfloor \frac{r_i}{2} \right \rfloor$ vertices for a total of $\left \lfloor \frac{r_i}{2} \right \rfloor+1$ colors.

Including the singletons, this implies that a total of at least
\[M \ge A+ \sum_{i=1}^D \left \lceil \frac{r_i}{2} \right \rceil\]
colors are used by the fixing move.
\subsection{Endpoint redundancy for $D \ge 3$}
We finally make a small algebraic observation which allows us to essentially ignore the second case.
\begin{lemma}
	For $D \ge 3$, 
	\[A+\left \lfloor \frac{A+R_{D-1}+1}{2} \right \rfloor \ge \min(F_{D-2},P) \]
\end{lemma}
\begin{proof}
Let $H =A+\left \lfloor \frac{A+R_{D-1}+1}{2} \right \rfloor $. The proof splits into two cases.

	\textit{Case 1 :} $A<R_{D-1}$. If $R_{D-2}<A$, then $F_{D-2}=2A+1$ by definition. Moreover, $R_{D-1} > A$ so $H \ge 2A+1 = F_{D-2}$. On the other hand, if $R_{D-2} \ge A$, then $F_{D-2}=2A+1+ \left \lceil \frac{R_{D-2}-A+1}{2} \right \rceil$ while
	\[H = 2A +\left \lfloor \frac{(R_{D-2}-A)+r_{D-1}+1}{2} \right \rfloor\ge F_{D-2}.\]
	since $r_{D-1} \ge 4$.

	\textit{Case 2 :} $A\ge R_{D-1}$. In this case, $H \ge A+ R_{D-1}$ and for $D \ge 3$, 
	\[\sum_{i=1}^D  \left \lceil \frac{r_i}{2} \right \rceil \le \sum_{i=1}^{D-1}  \left \lceil \frac{r_i}{2} \right \rceil +  \left \lfloor \frac{r_{D-1}}{2} \right \rfloor +  \left \lfloor \frac{r_{D-2}}{2} \right \rfloor  \le \sum_{i=1}^{D-1}  \left \lceil \frac{r_i}{2} \right \rceil + \sum_{i=1}^{D-1}  \left \lfloor \frac{r_i}{2} \right \rfloor = R_{D-1}\]
	Thus, 
	\[H \ge A+R_{D-1} \ge P.\]
\end{proof}
\subsection{Lower bound} Combining the observations of the last four sections, Bob can always ensure that at least the minimum among the described lower bounds is forced. Thus, 
\[\chi_g(K_{1^A,r_1,\dots , r_D}) \ge \min\{F_0,F_1,\dots , F_{D-2},P\}\]
for all $D \ge 3$.
\section{Results}
We can now put all our work together.
\subsection{$D \ge 3$}
\begin{theorem}
	For all $D \ge 3$, 
	\[\chi_g(K_{1^A,r_1,\dots , r_D}) = \min \{F_0,F_1,\dots , F_{D-2},P\}.\]
\end{theorem}
\begin{proof}
	Follows from sections 2.3 and 3.5.
\end{proof}
We are finally left to deal with the degenerate cases $D=1$ and $D=2$.
\subsection{$D=1$ and $D=2$}
\begin{theorem}
	For $D=1$, 
	\[\chi_g(K_{1^A,r}) = \begin{cases}
		\min \left\{2A-1, A+ \left \lceil \frac{r}{2} \right \rceil \right \} & A>1 \\
		2 & \text{otherwise}
	\end{cases}
	\]
\end{theorem}
\begin{proof}
Clearly, $\chi_g(K_{1,r})=1$ so we deal with $A>1$ in what follows.

	If Alice immediately rushes starting with the singletons, it takes her at most $A-1$ moves to open every part, since Bob is forced to open a part on his first move. Thus, $2A-1$ colors are sufficient. If Alice uses \textbf{(A4)}, as before, $ A+ \left \lceil \frac{r}{2} \right \rceil$ colors suffice.

If the large part is partially colored, Bob never played in a singleton, Alice colored all $A$ of them. Therefore, Alice introduced at least $A$ colors, while Bob does at least $A-1$ moves for a total of at least $2A-1$ colors. Moreover, if the large part is complete, as before $ A+ \left \lceil \frac{r}{2} \right \rceil$ colors are required.
\end{proof}
\begin{theorem}
	For $D=2$, 
	\[\chi_g(K_{1^A,r_1,r_2}) = \begin{cases}
		2A & A=r_1=r_2 \text{ odd} \\
		\min\{2A+1,P\} & \text{ otherwise}
	\end{cases}\]
\end{theorem}
\begin{proof}
	The proof that $\chi_g(K_{1^A,r_1,r_2}) \ge \min \{2A+1,H,P\}$ follows from our work in section 3. Also, the proof that $2A+1$ or $P$ colors are sufficient follows from our work in section 2.

	The missing component is to return to section 3.4 and determine when $H$ is not redundant for $D=2$. Note that if $H<2A+1$, we have $r_1 \le A$. Therefore, if $H<P$,
	\[\left \lceil \frac{A+r_1}{2} \right \rceil \ge r_1 \ge  \left \lceil \frac{r_1}{2} \right \rceil + \left \lceil \frac{r_2}{2} \right \rceil -1\]
	which since $A \ge r_1 \ge r_2$, forces equality and hence $A=r_1=r_2$, and the common value must be odd. Finally, in this case $H=2A$.

	To show that it is possible for Alice to reach the fixing move using at most $2A$ colors for $K_{1^A,A,A}$, consider an adjusted strategy for Alice.
	\begin{definition}
		Let \textbf{(A6)} denote the following strategy for Alice:
		\begin{enumerate}
			\item if one of the large parts has exactly $A-1$ colored vertices and the other is uncolored, color the last vertex of that part using an already present color
			\item otherwise, if an uncolored singleton exists, pick one
			\item otherwise, rush
		\end{enumerate}
	\end{definition}
To finish from here, we have two cases.

	\textit{Case 1 :} If the completion move (1) occurs, before that move Alice has only colored singletons. If $t$ moves have been performed by Bob and of these $s$ involve singletons, the large part contains $A-1=t-s$ colored vertices, while $t+s=A-1 +2s$ singleton vertices have been colored. Since there are only $A$ singletons, we are forced to have $s=0$. Thus, Bob never colors a singleton. 

	Moreover, exactly one singleton and the other large part remain untouched. The completion costs no new color, and the next two openings give a total of at most $2A$.

	\textit{Case 2 :} If the completion move never occurs, every Alice move opens a part, and Bob must have opened at least two parts. Alice therefore opens at most $A$ parts, Bob receives at most $A$ moves, for a total of at most $2A$ colors.  
\end{proof}
\section{General remarks and Open problems}
We have been able to confirm that the strategy \textbf{(B1')} --- which in this case degenerates to \textbf{(B1)} --- from \cite{Obszarski2023} is optimal for the considered class of graphs, while the more general conjecture of its optimality for all multipartite graphs remains open.

When we allow for the appearance of parts of size two and three, the problem becomes more complicated. The approach we outlined above does not extend naturally to these cases because the monotonicity arguments for $F_j$ only work under the assumption that $r_j \ge 4$ for all non-singleton parts. 

However, the present result suggests that the fixing move is the natural state at which to analyze the coloring game on complete multipartite graphs; extending this viewpoint to graphs containing parts of size two and three appears to be the next step toward a complete classification.

In particular, including parts of size three seems a promising next direction.
\begin{problem}
	Find the exact formula for $\chi_g$ for all complete multipartite graphs $K_{1^A,3^C,r_1,\dots , r_D}$ where $r_1 \ge r_2 \ge \dots \ge r_D \ge 4$.
\end{problem}
\section*{Declaration of generative AI and AI-assisted technologies}
No AI models, or automated proof assistants driven by generative AI were utilized in generating the mathematical results, proofs, or text of this paper. 
\bibliographystyle{plain}
\bibliography{refs}
\end{document}